\documentclass{amsart}
\usepackage[german, american]{babel}
\usepackage{amssymb,enumerate,verbatim,xspace,color}
\usepackage[latin1]{inputenc}

\numberwithin{equation}{section}

\theoremstyle{plain}
\newtheorem{theorem}{Theorem}[section]

\newtheorem{corollary}[theorem]{Corollary}
\newtheorem{fact}[theorem]{Fact}
\newtheorem{lemma}[theorem]{Lemma}
\newtheorem{prop}[theorem]{Proposition}

\theoremstyle{definition}
\newtheorem{remark}[theorem]{Remark}

\newtheorem{definition}[theorem]{Definition}

\newtheorem{convention}[theorem]{Convention}

\newcommand{\nc}{\newcommand}
\nc{\K}{\ensuremath{\mathcal{K}}\xspace}
\nc{\T}{\ensuremath{T^\mu}\xspace}
\nc{\C}{\ensuremath{\mathbb{C}}\xspace}
\nc{\Cn}[1]{(K^*)^{#1}}
\nc{\Gm}{\mathbb{G}_{\mathrm{m}}}
\nc{\inv}[1]{\cdot {#1}^{-1}}
\nc{\Q}{\ensuremath{\mathbb{Q}}\xspace}
\nc{\U}{\operatorname{\ddot{U}}}
\nc{\Ur}{\operatorname{\mathrm{U}}}
\nc{\Th}{\operatorname{Th}}
\nc{\GrTor}{\nu}
\nc{\Lan}{\mathcal{L}}

\nc{\D}{\mathcal{D}}
\nc{\stab}{\operatorname{stab}}
\renewcommand{\phi}{\varphi}
\nc{\acl}{\operatorname{acl}} 
\nc{\acle}{\operatorname{acl\e}}
\nc{\dcl}{\operatorname{dcl}} 
\nc{\dcle}{\operatorname{dcl\e}}
\nc{\dd}{\operatorname{d}} 
\nc{\cl}{\operatorname{cl}}
\nc{\erz}[1]{\langle #1\rangle} 
\nc{\eps}{\varepsilon}
 \nc{\ra}{r(\alpha)}
\nc{\ma}{m_\alpha}
\nc{\na}{n_\alpha}
\nc{\ka}{k_{\alpha}}
\nc{\eq}{\stackrel{.}{=}}
\nc{\tp}{\operatorname{tp}}
\nc{\stp}{\operatorname{stp}}
\nc{\alg}{\mathrm{alg}}
\nc{\ACF}{\mathrm{ACF}}
\nc{\ldim}{\operatorname{l.dim}}
\nc{\tr}{\operatorname{dim}}
\nc{\trdeg}{\operatorname{tr.d}}

\newcommand{\Tor}{\operatorname{Tor}}
\newcommand{\strong}{\leq}
\nc{\cd}{\operatorname{cd}}
\nc{\x}{\bar x}
\nc{\y}{\bar y}
\nc{\z}{\bar z}
\nc{\e}{\bar e}
\nc{\f}{\bar f}
\nc{\m}{\bar m}
\nc{\ba}{\bar a}
\nc{\bb}{\bar b}
\nc{\bg}{\bar g}
\nc{\Land}{\bigwedge}
\nc{\Lor}{\bigvee}

\newcommand{\Gl}[1]{\mathrm{GL}_{\mathrm{#1}}(\Q)}

\newcommand{\Clc}{\mathcal{C}}
\newcommand{\Cl}{\mathcal{K}}
\nc{\Clm}{\ensuremath{\K^\mu}\xspace}
\newcommand{\Z}{\mathbb{Z}}
\newcommand{\N}{\mathbb{N}}
\newcommand{\Mr}{\operatorname{MR}}

\def\Ind#1#2{#1\setbox0=\hbox{$#1x$}\kern\wd0\hbox to 0pt{\hss$#1\mid$\hss}
\lower.9\ht0\hbox to 0pt{\hss$#1\smile$\hss}\kern\wd0}
\def\Notind#1#2{#1\setbox0=\hbox{$#1x$}\kern\wd0\hbox to 0pt{\mathchardef
\nn="3236\hss$#1\nn$\kern1.4\wd0\hss}\hbox to 0pt{\hss$#1\mid$\hss}\lower.9\ht0
\hbox to 0pt{\hss$#1\smile$\hss}\kern\wd0}

\def\nind{\mathop{\mathpalette\Notind{}}}

\begin{document}

\bibliographystyle{plain}

\title{Bad fields with torsion}
\author{Juan Diego Caycedo} 
\address{Mathematisches Institut, Albert-Ludwigs-Universit\"at Freiburg, Eckerstr. 1, 79104 Freiburg, Deutschland.}
\email{juan-diego.caycedo@math.uni-freiburg.de}

\author{Martin Hils} 
\address{Univ Paris Diderot, Sorbonne Paris Cit\'e, UMR 7586, Institut de Math\'ematiques de Jussieu -- Paris Rive Gauche, B\^{a}timent Sophie Germain, Case 7012, F-75213 Paris, France.\newline\indent 
\'Ecole Normale Sup\'erieure, D\'epartement de math\'ematiques et applications, 45 rue d'Ulm, 75230 Paris Cedex 05, France (UMR 8553 du CNRS).}
\email{hils@math.univ-paris-diderot.fr}
\thanks{The second author was partially supported by ANR-13-BS01-0006-01 ValCoMo and by ANR-09-BLAN-0047 Modig, which also funded a visit of the first author to Paris during which part of this work was carried out.}

\date{\today}
\def\abstractname{Abstract}
\begin{abstract}We extend the construction of bad fields of characteristic zero to the case of arbitrary prescribed divisible green torsion.\end{abstract}

%%%
\maketitle

\section{Introduction}
In this note, we construct bad fields in characteristic 0 with arbitrary prescribed divisible green torsion. 
For this, we prove that the free amalgamation class is axiomatisable and show that the collapse 
to a bad field may then be performed exactly as in \cite{BaHiMaWa09}. 

The axiomatisability of the free amalgamation class was first proved in the doctoral thesis of the first author (\cite{CayThesis}), also in the case where the green points form a subgroup of an elliptic curve and with any finite-rank green subgroup in place of the green torsion. We include the proof in the relevant case, written in a way that is consistent with the presentations in \cite{Poi01} and \cite{BaHiMaWa09}.

In the part on the collapse we will rely heavily on the results of \cite{BaHiMaWa09} (and their proofs), only indicating at some key steps the necessary changes when allowing green torsion points. 
Roche had observed a gap in the proof of the collapse in \cite{BaHiMaWa09}, related to choices of green roots. The second author addressed this issue in \cite{Hil12}, showing that Kummer 
genericity is a definable property and proposing improved codes which take into account 
Kummer genericity. In the present paper, we seize the opportunity to spell out the points in the construction of the bad fields where the use of the improved codes is essential.

Recall that a \emph{bad field} is a field of finite Morley rank equipped with a definable proper infinite subgroup 
of the multiplicative group of the field. Their study originated in connection with the Cherlin-Zilber Algebraicity Conjecture which asserts that a simple (infinite) group of finite Morley rank must be an algebraic group. 

By a result of Wagner \cite{Wag03} the existence of a bad field of characteristic $p>0$ is very unlikely, as it would imply that there are only finitely many $p$-Mersenne 
primes, i.e.\ primes of the form $\frac{p^n-1}{p-1}$. Baudisch, Martin-Pizarro, Wagner and the second author constructed a bad field of characteristic 0, thus answering a long-standing 
open question of Zilber. More precisely, it is shown in \cite{BaHiMaWa09} that Poizat's \emph{green 
field}, an infinite rank analogue of a bad field which Poizat had obtained in \cite{Poi01} using Hrushovski's 
amalgamation method, may be collapsed to a structure of finite Morley rank. 

Following Poizat, we call the elements from the distinguished
multiplicative subgroup the \emph{green points} of the field.  In
Poizat's green fields, as well as in the bad fields from
\cite{BaHiMaWa09}, the green points form a divisible torsion-free
group. In the case of arbitrary prescribed divisible green torsion $\GrTor$,
Poizat shows that the construction can be carried out provided the
free amalgamation class is axiomatisable, and he proves that this
is the case under the assumption that the CIT, a conjecture of Zilber
on unlikely intersections, holds. We show that the free amalgamation class is axiomatisable for arbitrary prescribed divisible green torsion, unconditionally. In the proof, the use of the CIT is replaced by applications of a weaker proven statement, known as the Weak CIT, and a theorem of Laurent.

An overview of the paper. In Section \ref{S:Axiom}, we prove the axiomatisability of the free amalgamation class. We then gather the material needed from \cite{BaHiMaWa09}, with a particular emphasis on the construction of the improved codes, in Section \ref{S:Bad-Context}. The 
main results are given in Section \ref{S:Proof} where we construct both (infinite rank) green fields and bad fields with green torsion equal to $\GrTor$. We also include a complete proof of the axiomatisability of existential closedness, illustrating how the improved codes are used. 

The first author would like to thank Martin Bays, Zoé Chatzidakis, Philipp Hieronymi and Boris Zilber for very helpful comments on the topic of Section~\ref{S:Axiom}.

\section{Axiomatisability of the free amalgamation class} \label{S:Axiom}
Let $\Tor$ denote the group of roots of unity in $\Q^{\alg}$. Let us fix a divisible subgroup $\GrTor$ of $\Tor$.

Let $\Lan$ be the expansion of the language of rings by a unary predicate $\U$. Let $\Cl$ be the class of $\Lan$-structures $(K,+,-,\cdot,0,1,\U)$ satisfying the following conditions:
\begin{enumerate}[(i)]
\item $(K,+,-,\cdot,0,1)$ is an algebraically closed field of characteristic 0,
\item $\U$ is a divisible subgroup of $(K^*, \cdot)$,
\item the group of torsion elements of $\U$ is isomorphic to $\GrTor$,
\item for all $n \geq 1$ and all $y \in \U^n$, the value $\delta(y):= 2 \trdeg(y) - \ldim(y)$ is non-negative. 
\end{enumerate}

Here, $\trdeg(y)$ denotes the transcendence degree of the field $\Q(y)$ over $\Q$, and $\ldim(y)$ denotes linear dimension of the subspace generated by $y$ in the $\Q$-vector space $K^*/\Tor$.

We show below that $\Cl$ is an elementary class. In \cite{Poi01} the same result is proved assuming the Conjecture on Intersections with Tori (CIT) (cf. \cite[Corollaire 3.5]{Poi01}) and unconditionally only in the case where $\GrTor$ is trivial. The idea of replacing the use of the CIT by a combination of the Weak CIT and Laurent's theorem, Facts \ref{WCIT} and \ref{ML} below, comes from \cite{Zil11}.

In the definitions and facts below, $K$ denotes an algebraically closed field of characteristic 0.

%%%

Let us fix some notation. For every $n$, $(K^*)^n$ is an algebraic group under coordinate-wise multiplication. For $m = (m_1,\dots,m_n) \in \Z^n$ and $y = (y_1,\dots,y_n) \in (K^*)^n$, let  
$y^m := \prod_{i} y_i^{m_i}$. 
More generally, for a $k \times n$-matrix $M$ over $\Z$ with rows $M_1, \dots, M_k$ and $y = (y_1,\dots,y_n) \in (K^*)^n$, let $y^M := (y^{M_1}, \dots, y^{M_k})$.
The map from $(K^*)^n$ to $(K^*)^k$ given by $y \mapsto y^M$ is a homomorphism of algebraic groups. Its kernel, the subset of $(K^*)^n$ defined by the system of equations $y^M = 1$, is thus an algebraic subgroup of $(K^*)^n$. All algebraic subgroups of $(K^*)^n$ are of this form. If the matrix $M$ has rank $k$, then the corresponding algebraic subgroup $T$ has dimension $n-k$ as a Zariski closed set. For the dimension of $T$ we may use the notation $\dim (T)$ or $\ldim (T)$. 
A \emph{torus} is a connected algebraic subgroup of $(K^*)^{n}$.

%%%

A variety $V$ will always be a closed algebraic subvariety of some $(K^*)^{n}$ (not necessarily irreducible). 
Given an irreducible subvariety $V$ of $(K^*)^{n}$, its
\emph{minimal torus} is the smallest torus $T$, such that $V$ lies in
some coset $\ba\cdot T$.  In this case, we define
the \emph{codimension} of $V$ by $\cd(V):=\dim(T)-\dim(V)=\ldim(V)-\dim(V)$,
where $\ldim(V):=\dim(T)$. An irreducible subvariety $W\subseteq V$ is
called \emph{$\cd$-maximal in $V$} if $\cd(W')>\cd(W)$ for every irreducible variety $W'$ 
such that $W\subsetneq W'\subseteq V$.

%Recall that any connected algebraic subgroup of a torus is again a torus.

We now state a variant of Zilber's \emph{Weak CIT} \cite{Zil02} due to Poizat \cite[Corollaire 3.7]{Poi01}:

\begin{fact} \label{WCIT} Let $V(\x,\z)$ be a uniformly definable family of varieties in $(K^*)^{n}$. There exists a finite collection of tori $\{T_0,\dots,T_s\}$, such that for every $\bb$ the minimal torus of every $\cd$-maximal subvariety of $V_{\bb}$ belongs to the collection $\{T_0,\dots,T_s\}$.
\end{fact} 

Henceforth we denote by $\mathcal{T}(V(\x,\z))$ a collection $\{T_0,\dots,T_s\}$ of tori as provided by the above fact, with $T_0=(K^*)^{n}$ and $T_1=\{1\}^n$. 

The following is Laurent's theorem from \cite{Laurent84} on intersections of subvarieties of $(K^*)^n$ with a finite rank subgroup, in the case where the subgroup is $\Tor^n$.

\begin{fact} \label{ML}
For every proper algebraic subvariety $W$ of $(K^*)^{n}$, there exist proper algebraic subgroups $H_1,\dots,H_r$ of $(K^*)^n$ and $\gamma_1,\dots,\gamma_r \in \Tor^n$ such that 
\[
W \cap \Tor^n = \bigcup_{j=1}^r \gamma_j (H_j \cap \Tor^n).
\]
\end{fact}

\begin{theorem}  \label{T:Torsion-Axioms}
The class $\Cl$ is elementary.
\end{theorem}
\begin{proof}
It is clear that conditions (i) and (ii) can be expressed by a set of $\Lan$-sentences. Moreover, it is easy to see that condition (iii) can be expressed by a set of sentences requiring that $\U$ has non-trivial $p$-torsion precisely for those primes $p$ for which $\GrTor$ has non-trivial $p$-torsion.
 
We shall now see that, modulo (i),(ii) and (iii), condition (iv) is equivalent to the validity of the following sentences: for each $n \geq 1$, and each subvariety $W$ of $\Cn{n}$ defined and irreducible over $\Q$ of dimension $< \frac{n}{2}$, the sentence
\[
\psi_W := \forall y \bigg[
(y \in W \land y \in \U^n \land y \not\in W^*)  
\longrightarrow \bigvee_{\substack{1\leq i \leq s\\ W_{i} \neq (K^*)^{k_i}}} \bigvee_{1 \leq j\leq r_i} y^{M_i} \in H_{ij}
\bigg]
\]
where
\begin{itemize}
\item $T_1,\dots,T_s$ are the (proper) subtori of $\Cn{n}$ provided by Fact~\ref{WCIT} for $W$, and for $i=1,\dots,s$, $M_i$ is a $k_i \times n$-matrix over $\Z$ of rank $k_i$ such that $T_i$ is defined by the system of equations $y^{M_i} = 1$;
\item for each $i=1,\ldots s$, $W_i$ is the Zariski closure of the set $W^{M_i} := \{ y^{M_i} \mid y \in W \}$ inside $(K^*)^{k_i}$;
\item for each $i=1,\dots,s$ such that $W_i$ is a proper subvariety of $(K^*)^{k_i}$, for each 
$1\leq j\leq r_i$, $H_{ij}$ is a proper algebraic subgroup of $(K^*)^{k_i}$ defined over $\Q$ such 
that $H_{ij}\supseteq\gamma_{ij}H'_{ij}$, where $\gamma_{i1},\dots,\gamma_{ir_i}$ and $H'_{i1},\dots,H'_{ir_i}$ are as provided by Fact~\ref{ML} for the variety $W_i$;
\item $W^* := \bigcup_{i=1}^s W^{*i}$, where
$$W^{*i} := \{ b \in W\ \mid \dim W \cap bT_i > \dim W - \dim W_i \}.$$
\end{itemize}

%%%%%%%%%%%%%%%%%%%%%%%%%

Observe that each $W_i$, each $W^{*i}$, $W^*$ and the Zariski closure of $W^*$ are definable over $\Q$ (in the language of rings). 
Let us see that $W^*$ is a non-generic subset of $W$, i.e. its Zariski closure is not the whole of $W$. For each $i$, consider $f_i:W\rightarrow W_i$ given by $b\mapsto b^{M_i}$. Suppose $b$ is generic in $W$ over $\Q$. Then $f_i(b)$ is generic in $W_i$ over $\Q$. 
Also, $\dim W \cap bT_i = \trdeg(b/f_i(b))$. Indeed, $W \cap bT_i$ is defined over $f_i(b)$, as the fibre of $f_i$ above $f_i(b)$ inside $W$, and for any element $b'$ of $W \cap bT_i$ the element $f_i(b') = f_i(b)$ is algebraic over each of $b$ and $b'$, so $\trdeg(b'/f_i(b)) \leq \trdeg(b/f_i(b))$. 
Since $\trdeg(b/f_i(b)) =  \trdeg(b) - \trdeg(f_i(b))$, 
we get $\dim W \cap bT_i = \dim W - \dim W_i$. Thus, $b\not\in W^*$. 

\smallskip

Now assume $(K,\U)$ satisfies (i)-(iii) and all the above sentences $\psi_W$. To see that $(K,\U)$ must then also satisfy (iv), suppose towards a contradiction that $b$ is an $n$-tuple from $K$ such that $\delta(b) < 0$. It is easy to see that we may assume $b$ to be green and multiplicatively independent. Let $W$ be the algebraic locus of $b$ over $\Q$. Then, since $\delta(b) < 0$, we have $\dim W < \frac{n}{2}$. Thus, by our assumption, $\psi_W$ holds. Since $b$ is generic in $W$ over $\Q$, we have $b \in (W \setminus W^*)\cap \U^n$. Thus, we get a multiplicative dependence on $b$ from $\psi_W$, hence a contradiction. This proves that $(K,\U)$ satisfies (iv).

%%%%%%%%%%%%%%%%%%%%%%%%%

Conversely, assume that $(K,\U)$ satisfies (i)-(iv) and let us see that the above sentences hold in $(K,\U)$. Let $n \geq 1$ and let $W$ be a subvariety of $(K^*)^n$ defined and irreducible over $\Q$ of dimension $<\frac{n}{2}$. Suppose $b$ is in the set $(W \setminus W^*) \cap \U^n$. Since $\trdeg(b) \leq \dim W <n/2$ and by assumption $\delta(b) \geq 0$, the tuple $b$ must be multiplicatively dependent. Thus, we may choose $H$ a proper algebraic subgroup of $\Cn{n}$ containing $b$ 
such that $\dim(H)=\ldim(b)$. 

Let $T$ be the connected component of $H$ (so $T$ is a torus), and let $S$ be an irreducible component of 
$W \cap bT$ containing $b$. Note that $bT$ is defined over $\Q^{\alg}$, and so $S$ is defined over $\Q^{\alg}$ as well. In particular, there is $b'\in S(\Q^{\alg})$.

Note that (iv) implies that $\U\cap\Q^{\alg}=\GrTor$. It follows that $\ldim(b)=\ldim(b/\Q^{\alg})=\dim(T)$, and thus $0 \leq \delta(b) \leq 2\dim(S)-\dim(T)=\dim(S)-\cd(S)$, since $\trdeg(b) \leq \dim(S)$. We thus have $\cd(S)\leq\dim(S)<\frac{n}{2}$.

Let $S\subseteq S'\subseteq W$ be such that $\cd(S')=\cd(S)$ and $S'$ is $\cd$-maximal in $W$. So 
the minimal torus of $S'$ is equal to $T_i$ for some $i$, with $T\subseteq T_i$. Moreover, since $\cd(S')\leq\dim(S')<\frac{n}{2}$, necessarily $i\neq 0$. 
Let $\alpha=b'^{M_i}=b^{M_i}$. The coordinates of $\alpha$ are from $\U \cap\Q^{\alg}$, so $\alpha\in\GrTor^{k_i}$ by what we said above. Thus, $\alpha \in W_i \cap \Tor^{k_i}$, and hence it is in one of the $H_{ij}$.

It remains to show that $W_i$ is a proper subvariety of $\Cn{k_i}$. First, note that 
$\frac{n}{2}>\cd(S)=\cd(S')=\dim(T_i)-\dim(S')\geq\dim(T_i)-\dim(W\cap bT_i)$. Using $b\not\in W^*$, 
we conclude by the following calculation:
\begin{eqnarray*}
\dim(W_i)&=&\dim(W)-\dim(W\cap bT_i)\\
&<&\frac{n}{2}-\dim(W\cap bT_i)\\
&=&\frac{n}{2}-\dim(T_i)+(\dim(T_i)-\dim(W\cap bT_i))\\
&<& \frac{n}{2}-(n-k_i)+\frac{n}{2}=k_i.
\end{eqnarray*}
\end{proof}

\section{Improved codes}\label{S:Bad-Context}\newcounter{Kode_enum}
In this section, we give a precise description of the improved codes proposed in \cite{Hil12}. As was mentioned in the introduction, Roche had found a gap in the original construction which made this improvement necessary.

For $V\subseteq\Cn{n}$ and $m\geq1$, the set $\{(a_1,\ldots,a_n)\in\Cn{n}\mid (a_1^m,\ldots,a_n^m)\in V\}$ will be denoted by $\sqrt[m]{V}$.

\begin{definition}\label{D:Kummer}
Let $V$ be an irreducible subvariety of $\Cn{n}$.
\begin{itemize} 
\item $V$ is called \emph{Kummer generic} 
if $\sqrt[m]{V}$ is irreducible for every $m\geq1$. 
\item $V$ is called \emph{free} 
if it is not contained in a coset of a proper subtorus of $\Cn{n}$.
\end{itemize}
\end{definition}

Part (1) of the following fact is due to Zilber \cite{Zil06}, and part (2) is 
due to the second author \cite{Hil12}. For a more conceptual proof of the corresponding result in 
arbitrary semi-abelian varieties, see \cite{BaGaHi13}.

\begin{fact}\label{F:Kummer}
\begin{enumerate}
\item If $V$ is free, then there is $N\geq 1$ such that every irreducible component of $\sqrt[N]{V}$ 
is Kummer generic.
\item Kummer genericity is a definable property, i.e.\ given a definable family $V(x,z)$ of subvarieties of $\Cn{n}$, the set of parameters $b$ for which $V(x,b)$ is Kummer generic is definable.
\end{enumerate}
\end{fact}

Before we construct the improved codes, let us recall the notion of minimal prealgebraicity from \cite[Section 3]{BaHiMaWa09}. 

For $A\subseteq K^*$, we denote by $\langle A\rangle$ the divisible hull of the 
(multiplicative) subgroup of $K^*$ generated by $A$. If $b\in K^*$ is a finite tuple, we let 
$\delta(\langle Ab\rangle/\langle A\rangle):=\delta(b/A):=2 \trdeg(b/A) - \ldim(b/A)$.

\begin{definition}
\begin{itemize}
\item Let $A=\langle A\rangle\subseteq \langle Ab\rangle=B$. The extension $B/A$ 
is called \emph{minimal prealgebraic 
(of length $n$)} 
if the following conditions hold:

\begin{itemize}
\item $2\leq \ldim(B/A)=n<\infty$, 
\item $\delta(B/A)=0$, and  
\item $\delta(B/B')<0$ for every $B'=\langle B'\rangle$ with $A\subsetneq B'\subsetneq B$.
\end{itemize}

\item A strong $n$-type $p(x)=\stp(b/A)$ is \emph{minimal prealgebraic} 
if the extension $\langle Ab\rangle/\langle A\rangle$ is minimal prealgebraic of length $n$ (in particular, in this case the tuple $b$ is multiplicatively 
independent over $A$).
\end{itemize}
\end{definition} 

For two formulas $\phi(x)$ and $\psi(x)$ of Morley degree 1 we write $\phi\sim\psi$ if 
$\Mr(\phi\Delta\psi)<\Mr(\phi)$, where $\phi\Delta\psi$ denotes their symmetric difference.

\begin{definition}\label{D:Kodevarietaet}
Let $\phi(x)$ be a formula of Morley degree 1 ($x$ ranging over $\Cn{n}$).
\begin{itemize}
\item $\phi(x)$ is called \emph{minimal prealgebraic} if its generic type is minimal prealgebraic.  
\item $\phi(x)$ is called \emph{Kummer generic} if the unique irreducible variety $V(x)$ such that 
$\phi\sim V$ is Kummer generic. Similarly, a strong type is called \emph{Kummer generic} if it is the generic type of a Kummer generic variety.
\end{itemize}
\end{definition} 

Let $T\subseteq\Cn{n}\times\Cn{n}$ be an $n$-dimensional torus such that $\pi_1(T)=\Cn{n}=\pi_2(T)$. Such a torus will be called a \emph{correspondence torus}.
Let $\phi_1$ and $\phi_2$ be two formulas of Morley degree 1, and let $X_i\subseteq\Cn{n}$ be 
the set defined by $\phi_i$.  
We say that $T$ induces a \emph{toric correspondence} between $\phi_1$ and $\phi_2$ 
if $(X_1\times X_2)\cap T$ projects generically onto both $X_1$ and $X_2$.

The following lemma is easy.
\begin{lemma}\label{L:toric}
\begin{enumerate}
\item The set of Kummer generic formulas is closed under multiplicative translations: if $\phi(x)$ is 
Kummer generic and $m\in\Cn{n}$, then $\phi(x\cdot  m)$ is Kummer generic, too.
\item The set of minimal prealgebraic formulas is closed under toric correspondences and under 
multiplicative translations.
\end{enumerate}
\end{lemma}

\begin{corollary}\label{C:code-enough}
\begin{enumerate}
\item Every minimal prealgebraic formula is in toric correspondence with some Kummer generic (and minimal prealgebraic) formula.
\item For every formula $\phi(x,z)$ the set of $b$ such that $\phi(x,b)$ is minimal prealgebraic 
(Kummer generic, respectively) is definable. 
\end{enumerate}
\end{corollary}

\begin{proof}
(1) Using Lemma \ref{L:toric}, this follows from Fact \ref{F:Kummer}(1), since every minimal prealgebraic 
variety is free.

(2) Minimal prealgebraicity is a definable property by \cite[Lemma 4.3]{BaHiMaWa09}, and 
the definability of Kummer genericity is Fact \ref{F:Kummer}(2).
\end{proof}

\begin{lemma} \label{C:Kg}
Let $V\subseteq \Cn{n}$ be a Kummer generic variety and $T\subseteq\Cn{2n}$ a 
correspondence torus. Then there is a unique (irreducible) variety $V'\subseteq\Cn{n}$ 
such that $T$ induces a toric correspondence between $V$ and $V'$.
\end{lemma}

\begin{proof}
Let $a=(a_1,\ldots,a_n)$ be generic in $V$ over $K$. As $\pi_1(T)=\Cn{n}$, there is $a'$ such that $(a,a')\in T$. Let $V'$ be the locus 
of $a'$ over $K$. Then $T$ induces a toric correspondence between $V$ and $V'$. This proves existence.

To prove uniqueness of $V'$, note that, modulo torsion, $T$ corresponds to the graph of some $\Gamma\in\Gl{n}$. 
Let $N\in\N$, $N>0$, be such that $\Gamma=\frac{1}{N}\Delta$ for a matrix $\Delta$ with integer coefficients. We may thus find a tuple 
$\alpha'$ so that $\alpha'^{\Delta}=a'$ and $(\alpha'_i)^{N}=a_i$ for $i=1,\ldots,n$. Moreover, whenever $T$ induces a toric correspondence 
between $V$ and some irreducible varitey $V''$, there is $a''$ such that $(a,a'')\in T$ and $a''$ is generic in $V''$ over $K$. As before, we may find $\alpha''$ such 
that $\alpha''^{\Delta}=a''$ and $(\alpha''_i)^N=a_i$ for $i=1,\ldots,n$. By Kummer genericity of $V$, $\tp(\alpha'/K)=\tp(\alpha''/K)$. In particular, 
$\tp(a'/K)=\tp(a''/K)$, proving that $V'$ is unique.
\end{proof}

If one does not assume $V$ to be Kummer generic, $V'$ is in general not unique. 

\begin{definition}\label{D:setcod}Let $X\subseteq \Cn{n}$ be definable. A formula $\phi(x,z)$ and a torus $T$ \emph{encode} $X=X(x)$, if there is some $b$ such that $T$ induces a toric correspondence between $\phi(x,b)$ and $X(x)$. We say that $\phi$ \emph{encodes} $X$ if the above correspondence is the identity (i.e., $\phi(x,b)\sim X$).
\end{definition} 

\begin{definition}\label{D:Kode}A \emph{code} $\alpha$ is a 
$\emptyset$-definable formula $\phi_{\alpha}(x,z)$ and integers 
$\na$, $\ka$ satisfying the following: 
\begin{enumerate}[(a)] 
\item The length of $x$ is $n_\alpha=2k_\alpha$.
\item $\phi_{\alpha}(x,b)$ is a subset of $\Cn{n_\alpha}$.
\item\label{C:Kode_grad} $\phi_{\alpha}(x,b)$ is either empty or has Morley rank $k_\alpha$ and Morley degree $1$. 
\item\label{C:prealg}If $\phi_\alpha(x,b)\neq\emptyset$, then 
$\phi_\alpha(x,b)$ is minimal prealgebraic and Kummer generic, with irreducible Zariski closure $V_\alpha(x,b)$. 
\item\label{C:sm} Suppose $\phi_{\alpha}(x,b)\neq\emptyset$. Then
$\delta(a/B)\leq0$ for every $b\in B$ and
$a\models\phi_{\alpha}(x,b)$. Moreover, $\delta(a/B)=0$ if and only if 
$a\in\langle B\rangle$ or $a$ is $B$-generic in $\phi_{\alpha}(x,b)$.
\item\label{C:trans} $\phi_{\alpha}(x,z)$ encodes every multiplicative translate of $\phi_{\alpha}(x,b)$. 
\item\label{C:KanPar}If $\emptyset\neq\phi_{\alpha}(x,b)\sim\phi_{\alpha}(x,b')$, then 
$b=b'$.\setcounter{Kode_enum}{\value{enumi}}\end{enumerate}\end{definition}

We set $\theta_\alpha(z):=\exists x\phi_\alpha(x,z)$.

If follows from (\ref{C:KanPar}) that $b$ is the canonical base of the minimal prealgebraic type determined by $\phi_{\alpha}(x,b)$. Note that the only place where this definition differs from the 
one given in \cite[4.7]{BaHiMaWa09} is (\ref{C:prealg}), where we added Kummer genericity as a 
condition.

\begin{lemma}\label{L:prealgenKoded}Let $X$ be a minimal prealgebraic and Kummer generic 
definable set. Then $X$ can be encoded by some code $\alpha$.
\end{lemma} 

\begin{proof}Using definability of Kummer genericity (Corollary \ref{C:code-enough}(2)), it is easy to see that the proof of 
\cite[Lemma 4.8]{BaHiMaWa09} adapts to our context.
\end{proof}

Combining Corollary \ref{C:code-enough}(1) with Lemma \ref{L:prealgenKoded}, one sees 
that the proof of \cite[Satz 4.10]{BaHiMaWa09} goes through, yielding the following result.

\begin{prop}\label{P:Good-Codes}There exists a collection $\Clc$ of codes with the following 
properties:
\begin{enumerate}
\item Every minimal prealgebraic definable set $X$ can be encoded by some $\alpha\in\Clc$ and 
some correspondence torus $T$.
\item The code $\alpha\in\Clc$ from (1) is uniquely determined by $X$, and there are only finitely many 
correspondence tori $T$ such that $X$ is encoded by $\alpha$ and $T$.
\end{enumerate}
\end{prop} 

For the rest of the paper, we fix a set of codes $\Clc$ satisfying the conclusion of Proposition 
\ref{P:Good-Codes}, and we call the elements of $\Clc$ \emph{good codes}.

\begin{fact}[{\cite[Lemma 4.9]{BaHiMaWa09}}] \label{F:Galpha}
Let $\alpha\in\Clc$ and let $G(\alpha,\alpha)$ be the set of correspondence tori $T$ that induce 
a toric correspondence between some non-empty instances $\phi_\alpha(x,b)$ and 
$\phi_\alpha(x,b')$ of $\alpha$. Then $G(\alpha,\alpha)$ is finite.
\end{fact}

\subsection{Difference sequences}\label{Sub:Diff}
We now recall an important technical device from \cite{BaHiMaWa09} which is used in the collapsing process. 

\smallskip

For a good code $\alpha\in\Clc$, we choose $\ma<\omega$ such that for any $b\models\theta_\alpha$, 
$b$ is definable over any Morley sequence in $\phi_\alpha(x,b)$ of length $\ma$. (The existence 
of $\ma$ follows from code property (\ref{C:KanPar}).)

\begin{fact}[\mbox{\cite[Satz 5.5]{BaHiMaWa09}}]
For any code $\alpha\in\Clc$ and any $\lambda\geq\ma$ there is a formula 
$\psi_{\alpha}(x_0,\ldots,x_\lambda)$ (whose realisations are called \emph{difference sequences}) satisfying the following properties:

\smallskip

\begin{enumerate}[\upshape(a)]\setcounter{enumi}{\value{Kode_enum}} 
\item\label{D:versch} If $\models\psi_{\alpha}(\e_0,\ldots,e_\lambda)$, then $e_i\neq e_j$ for $i\neq j$. 
\item\label{D:MF} For $b\models\theta_\alpha$ and any Morley sequence 
$(e_0,\dotsc,e_{\lambda},f)$ in $\phi_{\alpha}(x,b)$, we have $$\models\psi_{\alpha}(e_0\inv{f},\dotsc,e_{\lambda} \inv{f}).$$

 \item\label{D:Kode_dcl} If $\models\psi_\alpha(e_0,\dotsc,e_{\lambda})$, there exists 
 a unique $b$ with $\models\phi_{\alpha}(e_i,b)$ for $i=0,\ldots,\lambda$, called the \emph{canonical 
 parameter} of the sequence. Moreover, $b$ lies in the definable closure of any subsequence of $(e_0,\ldots,e_\lambda)$ of length $\ma$.  
 \item\label{D:kleiner} If $\models\psi_{\alpha}(e_0,\dotsc,e_{\lambda})$, then $\models\psi_\alpha(e_0,\ldots,e_{\lambda'})$ for each  $\ma\leq\lambda'<\lambda$.  
\item\label{D:Kode_fork} Let $i\neq j$ and let $b$ be the canonical parameter of the sequence 
$(e_0,\dotsc,e_{\lambda})\models\psi_\alpha$. If there is some $T$ in $G(\alpha,\alpha)$ 
and $e_j'$ with $(e_j,e_j')\in T$ and if $e_i$ is generic in $\phi_\alpha(x,b)$ then 
$e_i\nind_{b}\
e_j'\inv{e_i}$.
\item\label{D:diff} If $\models\psi_{\alpha}(e_0,\dotsc,e_{\lambda})$, then 
$\models \psi_{\alpha}(\partial_i(e_0,\ldots,\e_{\lambda}))$ for 
$i\in \{0,\dots,\lambda\}$, where $$\partial_i(e_0,\ldots,e_\lambda):=(e_0\inv{e_i},\dotsc,e_{i-1}\inv{e_i},e_i^{-1},e_{i+1}\inv{e_i},
\dotsc,e_{\lambda}\inv{e_i}).$$
\end{enumerate}\end{fact}

\section{Main results and elements of the proof}\label{S:Proof}
In this last section, we will return to the coloured context of Section \ref{S:Axiom} and 
state our main results, both in the uncollapsed case (Theorem \ref{T:Uncollapsed}) and in the collapsed case (Theorem \ref{T:BadField}), thus constructing green fields of Poizat and bad fields with green torsion equal to 
$\GrTor$. Once the axiomatisability of the class $\Cl$ is established, the proofs 
of the corresponding results in the case were $\GrTor$ is trivial go through without major 
changes. In particular, the presence of green torsion does not affect the arguments in the collapsing process.

While we are at it, we will indicate the places where the use of the improved codes in the collapse 
is crucial. As an illustration, we will present a complete 
proof of the axiomatisability of existential closedness (Proposition \ref{P:8.4}). 

\subsection{Green colour}      \label{Sub:Color}
It is convenient to slightly 
modify the definition of the class $\Cl$ from Section \ref{S:Axiom}, allowing not only structures of the form $(K,+,\cdot,0,1,\U)$ satisfying 
conditions (i)-(iv) from the beginning of Section \ref{S:Axiom} but also $\langle\cdot\rangle$-closed 
subsets of such structures. We will do this working in a language $\Lan^*=\Lan_{\mathrm{Morley}}\cup\{\U\}$, where $\Lan_{\mathrm{Morley}}$ is a relational language in which $\ACF_0$ may be axiomatised 
and has quantifier elimination. (See \cite[Section 6]{BaHiMaWa09}.) Clearly, Theorem 
\ref{T:Torsion-Axioms} holds in this modified setting, i.e.,\ $\Cl$ is axiomatisable in $\Lan^*$. 

Let $B\subseteq A$ be structures from $\Cl$. If $\ldim(A)$ is finite, let $\delta(A)=2\trdeg(A)-\ldim(\U(A))$. 
If $\ldim(A/B)$ is finite, or more generally if both $\trdeg(A/B)$ and $\ldim(\U(A)/\U(B))$ are finite, we 
set $\delta(A/B))=2\trdeg(A/B)-\ldim(\U(A)/\U(B))$, the \emph{predimension} 
of $A$ over $B$. Note that if 
$A$ has a green linear basis over $B$, this definition coincides with the one used in the previous section. 

We say that 
$B$ is \emph{self-sufficient} in $A$ if $\delta(A'/B)\geq0$ for all $B\subseteq A'=\langle A'\rangle\subseteq A$ with $\ldim(A'/B)<\infty$. As usual we denote this by $B\leq A$. For the basic properties of $\delta$ and $\leq$, we refer to \cite{BaHiMaWa09}.

\begin{convention}\label{Conv}
\begin{itemize}
\item In the following, we use terms like \emph{algebraic}, \emph{generic}, \emph{Morley sequence}, $\tp(\cdot)$, $\stp(\cdot)$ etc. with respect to the theory $\ACF_0$. In particular, $a\in\acl(A)$ means that $a$ is algebraic over $A$ in 
the field sense.

\item An extension $B\leq A$ in $\Cl$ will be called \emph{minimal prealgebraic} if there is a green 
tuple $a\in A$ such that $\tp(a/B)$ is minimal prealgebraic and $\langle Ba\rangle=A$.

\item We will not distinguish 
between $\langle A\rangle\subseteq K^*$ and $\langle A\rangle\cup\{0\}\subseteq K$. 
This abuse of notation is entirely harmless. 
\end{itemize}
\end{convention}

Call a self-sufficient extension $B\leq A$ \emph{minimal} if it is proper and such that whenever $B\leq A'=\langle A'\rangle \leq A$, either $A'=B$ or $A'=A$.

\begin{fact}[\mbox{\cite[Lemma 6.4]{BaHiMaWa09}}]
Let $B\leq A$ be a minimal self-sufficient extension in $\Cl$. Then one of the following 
cases holds:
\begin{enumerate}
\item (\emph{algebraic}): $\U(A)=\U(B)$ and $A=\langle Ba\rangle$ for some element 
$a\in\acl(B)\setminus B$;
\item (\emph{white generic}): $\U(A)=\U(B)$ and $A=\langle Ba\rangle$ for some element 
$a\not\in\acl(B)$;
\item (\emph{green generic}): there is an element $a\in\U(A)\setminus\acl(B)$ such that 
$A=\langle Ba\rangle$;
\item (\emph{minimal prealgebraic}): $B\leq A$ is minimal prealgebraic (in the sense of 
Convention \ref{Conv}).
\end{enumerate}
\end{fact}

\begin{lemma}\label{L:Kummer-strong1}
Let $B=\langle B\rangle\in\Cl$ and let $p(x)$ be a strong field type over $B$ that is Kummer generic. 
For $i=1,2$, let $B\subseteq A_i=\langle Ba_i\rangle$, 
where $a_i\models p$ 
is a green $\Q$-basis of $\U(A_i)$ over $\U(B)$. Then $a_1\mapsto a_2$ extends to an isomorphism $A_1\simeq A_2$ over $B$.
\end{lemma}

\begin{proof}
For $i=1,2$, using the divisibility of $\U(A_i)$, construct inductively a sequence $(a_{i,n})_{n \geq 1}$ in $\U(A_i)$ such that $a_{i,n}^n = a_i$ and $(a_{i,mn})^m = a_{i,n}$ (componentwise) for all $m,n \geq 1$. Kummer genericity of $p(x)$ implies that for every $n\geq 1$, $\stp(a_{1,n}/B)=\stp(a_{2,n}/B)$,
so the map defined by $a_{1,n} \mapsto a_{2,n}$ for all $n \geq 1$ extends to an $\Lan^*$-isomorphism from $A_1$ to $A_2$ over $B$.
\end{proof}

\subsection{The uncollapsed case: green fields with green torsion equal to $\nu$}
It is shown in \cite{Poi01} 
that the class $(\Cl,\leq)$ has the amalgamation property and the joint 
embedding property. Moreover, using Lemma \ref{L:Kummer-strong1} together with 
Fact \ref{F:Kummer}(1), one may show that every $A\in\Cl$ 
of finite linear dimension has only countably many strong extensions of finite linear dimension, up to isomorphism. There 
is thus a unique countable rich $M_\omega\in\Cl$, the \emph{Fra\"{i}ss\'e-Hrushovski limit} of the 
subclass of $\Cl$ given by the structures of finite linear dimension. 
By definition, the theory of green fields, $T_\omega$, is the complete theory of the structure $M_\omega$. 
In \cite{Poi01}, Poizat notes that his results hold for arbitrary divisible green torsion $\GrTor$ if the axiomatisability of the class $\Cl$ can be established unconditionally. Since this missing point is provided by Theorem \ref{T:Torsion-Axioms}, we obtain the following result. 

\begin{theorem}\label{T:Uncollapsed}
$T_\omega$ is $\omega$-stable of Morley rank $\omega\cdot2$, with $\U$ having Morley rank $\omega$. The green torsion in models of $T_\omega$ is equal to $\GrTor$. 
\end{theorem}

A structure in $\Cl$ is rich if and only if it is an $\omega$-saturated model of $T_\omega$. It follows that the type of a $\langle\cdot\rangle$-closed self-sufficient subset of a model of 
$T_\omega$ is determined by its quantifier free type (in $\Lan^*$). 
% Recall that if 
% $B=\langle B\rangle\leq C$, $B\leq A=\langle A\rangle\subseteq C$ and $\delta(A/B)=0$, then 
% $A\leq C$. 
Using code properties 
(\ref{C:prealg}) and (\ref{C:sm}), Lemma \ref{L:Kummer-strong1} yields the following result.

\begin{lemma}\label{L:Kummer-strong2}
Let $\alpha\in\Clc$, $b\in B=\langle B \rangle \leq M\models T_\omega$ and $a_i\in M $ with 
$M\models\phi_\alpha(a_i,b)$ and $a_i\not\in B$ for $i=1,2$. 
Then $\tp_{T_\omega}(a_1/B)=\tp_{T_\omega}(a_2/B)$.\qed
\end{lemma}

The following remark, an immediate consequence of the lemma, clarifies the role of Kummer genericity and green torsion as far as multiplicity 
issues in $T_\omega$ are concerned. Indeed, the corresponding remark 
\cite[Bemerkung 6.7]{BaHiMaWa09} is incorrect as stated, and one needs to add Kummer genericity 
to the assumptions. 

\begin{remark}\label{R:sm}
Let $\phi_\alpha(x,z)$ 
be a code. Assume that $\models \theta_\alpha(b)$. Then $\phi_\alpha(x,b)\wedge\bigwedge_{i=1}^n \U(x_i)$ is a strongly minimal formula 
in $T_\omega$.

In fact, one may show that it would be enough to require in the definition of a Kummer generic variety 
$V$ that $\sqrt[p]{V}$ is irreducible for those primes $p$ for which $\GrTor$ does not contain 
any (non-trivial) $p$-torsion.
\end{remark}

\subsection{The collapse to a bad field}
In order to collapse the green field of Poizat we obtained to a bad field, we may proceed as in \cite[Sections 7--11]{BaHiMaWa09}. Alas, the procedure is quite technical. The idea is to forbid infinitely many realisations of the same minimal prealgebraic extension. The codes allow us to work uniformly in parameters, and the notion of difference sequences helps to address the delicate issue of controlling the interactions between different minimal prealgebraic extensions. 

\smallskip

We choose functions $\mu^*,\mu:\Clc\rightarrow\N$ with finite fibres satisfying some technical conditions, namely (1) $\mu^*(\alpha)\geq n_\alpha k_\alpha+1$, (2) $\mu^*(\alpha)\geq\lambda_\alpha(m_\alpha+1)$ and (3) $\mu(\alpha)\geq\lambda_\alpha(\mu^*(\alpha))$. Here, 
$\lambda_\alpha:\N\rightarrow\N$ is some strictly increasing function related to the conclusion of 
\cite[Lemma 7.3]{BaHiMaWa09}. We now define the class $\Clm$ as the subclass of $\Cl$ consisting of those $M$ which 
do not contain any green difference sequence of $\alpha$ of length $\mu(\alpha)+1$, for any 
good code $\alpha$.

\smallskip

In the proofs of the following two propositions (which correspond to \cite[Folgerung 8.4]{BaHiMaWa09} 
and \cite[Satz 9.2]{BaHiMaWa09}, respectively), Kummer genericity is used in an essential way. 
At the end of this section, we will present the argument for one of them in detail, namely for Proposition \ref{P:8.4}, since in this case the use of Kummer genericity is less apparent.

\begin{prop} \label{P:8.4}
For every good code $\alpha$, there is a $\forall\exists$-sentence $\chi_\alpha$ such that for every algebraically closed structure $M$ in $\Cl^\mu$,  the sentence $\chi_\alpha$ holds in $M$ if and only if $M$ has no minimal prealgebraic extensions in $\Cl^\mu$ coded by $\alpha$.
\end{prop}

\begin{prop}\label{P:9.2}
The class $(\Clm,\leq)$ has the amalgamation property.
\end{prop}

Let $M^\mu$ be the Fra\"{i}ss\'e-Hrushovski limit of the class $(\Clm,\leq)$, i.e., the unique countable 
structure in $\Clm$ that is rich for the subclass of $\Clm$ of all structures of finite 
linear dimension. Set $T^\mu:=\mathrm{Th}(M^\mu)$. 

Consider the theory $\tilde{T}^\mu$ expressing the following for an $\Lan^*$-structure $M$:

\begin{enumerate}
\item $M\in\Clm$;
\item $M\models\ACF_0$;
\item $M\models\chi_\alpha$ for every good code $\alpha\in\Clc$;
\item axioms which guarantee that if $M$ is $\omega$-saturated, then there are elements 
$g_i\in \U(M)$, $i\in\N$, such that $\dd(g_1,\ldots,g_n)=n$ for every $n$.
\end{enumerate}

Here, $\dd$ denotes the usual dimension function associated to the predimension $\delta$ (see \cite[Definition 10.1]{BaHiMaWa09}).

The axiomatisability of (1) follows from Theorem \ref{T:Torsion-Axioms}, and the $\chi_\alpha$ in (3) are from Proposition \ref{P:8.4}. Finally, (4) is axiomatisable by \cite[Lemma 10.3]{BaHiMaWa09}.

The theory $\tilde{T}^\mu$ we just defined is an axiomatisation of $T^\mu$. Indeed, the following proposition is proved exactly as \cite[Satz 10.5]{BaHiMaWa09}.

\begin{prop}\label{P:Rich=saturated}
The $\omega$-saturated models of $\tilde{T}^\mu$ are precisely the rich structures in $\Clm$. 
In particular, $T^\mu=\tilde{T}^\mu$.
\end{prop}

As in the case without green torsion, the theory $T^\mu$ has some level of 
quantifier elimination (the quantifier free type of a $\langle\cdot\rangle$-closed self-sufficient set 
determines its type) and is model complete (cf. \cite[Folgerung 10.6 and 10.7]{BaHiMaWa09}). 

\smallskip

We may now state our main result, the analogue of \cite[Satz 11.2]{BaHiMaWa09}.

\begin{theorem}\label{T:BadField}
$T^\mu$ has Morley rank 2, with $\U$ being strongly minimal. In particular, $M^{\mu}$ is a bad field of rank 2 with green torsion $\U\cap\Tor=\GrTor$.
\end{theorem}

Indeed, as in \cite{BaHiMaWa09} one shows that $\Mr(a/B)=\dd(a/B)$ holds in any model of $T^\mu$.

\subsection{Proof of the axiomatisability of existential closedness}
We finish with the proof of Proposition \ref{P:8.4} on the axiomatisability of existential closedness. Here we follow closely the argument in \cite{BaHiMaWa09}. At 
two places we will make an essential use of the improved codes, 
namely when applying Lemma \ref{C:Kg} and Lemma~\ref{L:Kummer-strong2}.

The following lemma provides the key structural property used in the proof.

\begin{lemma}[\mbox{\cite[Folgerung 8.3]{BaHiMaWa09}}]\label{L:8.3}
Let $M\in\Cl^\mu$, and let $M\leq M'\in\Cl$ be a minimal self-sufficient extension of $M$.
\begin{enumerate}
\item Assume $M'/M$ is algebraic, green generic or white generic. Then 
$M'\in\Cl^\mu$.
\item Assume $M'/M$ is minimal prealgebraic. Then $M'\not\in\Cl^\mu$ if and only if there is a good 
code $\beta\in\Clc$ and a difference sequence  $(e_0,\dots,e_{\mu(\beta)})$ for $\beta$ in $M'$ such 
that one of the following two cases occurs:
\begin{enumerate}
\item $e_0,\dots,e_{\mu(\beta)-1}\in M$, $\langle Me_{\mu(\beta)}\rangle=M'$, and $\beta$ is the 
unique code which describes the extension $M'/M$.
\item A subsequence of $(e_0,\dots,e_{\mu(\beta)})$ of length $\mu^*(\beta)$ is a Morley sequence 
for $\phi_\beta(x,b)$ over $Mb$, where $b$ is the canonical parameter of the sequence.
\end{enumerate}
\end{enumerate}
\end{lemma}

\begin{proof}[Proof of Proposition \ref{P:8.4}]
Let $\alpha \in \Clc$. Suppose $M \in \Cl^\mu$ is algebraically closed, $b \in M$ and $a$ is a green generic solution of $\phi_\alpha(x,b)$ such that $M[a] := \langle Ma\rangle$ is not in $\Cl^\mu$. Hence there exists a good code $\beta$ and a difference sequence $(e_0,\dots,e_{\mu(\beta)})$ for $\beta$ in $M[a]$. Let $b'$ be the canonical parameter of the sequence.

By Lemma \ref{L:8.3}, we may assume that either (a) $e_0,\dots,e_{\mu(\beta)-1}$ are in $M$, 
$M[a] = M[e_{\mu(\beta)}]$ and $\beta = \alpha$, or (b) there is a subsequence of $(e_0,\dots,e_{\mu(\beta)})$ of length $\mu^*(\beta)$ that is a Morley sequence for $\phi_\beta(x,b')$ over $M b'$.

In case (a), $b'$ is in $M$ and, since $M \strong M[a]$, we have that $e_{\mu(\alpha)}$ is generic in $\phi_\alpha(x,b')$ over $M$ (and green).
It follows that both tuples $a$ and $e_{\mu(\alpha)}$ are $\Q$-linear bases of $\U(M[a])$ over $\U(M)$.

Consider therefore the following condition on $b$: 

\begin{itemize}
\item[($*$)] There exist an $n$-tuple $m\in\U(M)$, a torus 
$T \in G(\alpha,\alpha)$ and a difference sequence $(e_0,\dots,e_{\mu(\alpha)-1})$ for $\alpha$ 
such that 
$T$ induces a toric correspondence between $\psi_\alpha(e_0,\dots,e_{\mu(\alpha)-1},x)$ and  $\phi_\alpha(x \cdot m, b)$. 
\end{itemize}

By code property (\ref{C:trans}), $\phi_\alpha(x \cdot m, b)$ is coded by $\alpha$ as well. Thus, by 
the finiteness of $G(\alpha,\alpha)$ (Fact~\ref{F:Galpha}), ($*$) can be expressed 
by $\theta^{(a)}(b)$, where $\theta^{(a)}(z)$ is an existential formula. Moreover, 
the Kummer genericity of $\phi_\alpha(x\cdot m,b)$ ensures that if $a'$ is a 
generic green solution of $\phi_\alpha(x\cdot m,b)$ over $M$ (equivalently, $a'\cdot m^{-1}$ is a generic green solution of $\phi_\alpha(x,b)$) then for any green tuple $e$ with $(a',e)\in T$ one 
has $\psi_\alpha(e_0,\dots,e_{\mu(\alpha)-1},e)$ (by Lemma \ref{C:Kg}). This shows that $\theta^{(a)}(b)$ holds if and 
only if we are in case (a).

In case (b), since all elements in the Morley subsequence are linearly independent over $M$, we have $\mu^*(\beta) \leq \frac{n_\alpha}{n_\beta} \leq n_\alpha$. Since $\mu^*$ is finite-to-one, there are only finitely many good codes $\beta$ for which this happens.

Express the set defined by the formula $\psi_\beta$ as a finite union  $\bigcup_{k=1}^r (V_k \setminus Z_k)$ where $V_k$ and $Z_k$ are varieties defined over $\Q$ with $Z_k \subsetneq V_k$.

Let $V_0 = V_{\alpha}(x,b)$ be the Zariski closure of $\phi_\alpha(x,b)$. This is an irreducible variety by code property (\ref{C:prealg}), and so it is equal to the locus of $a$ over $M$. Fix $k \in \{1,\dots,r\}$ such that $(e_0,\dots,e_{\mu(\beta)})\in V_k\setminus Z_k$, and let $W$ be the locus of 
$(a,e_0,\dots,e_{\mu(\beta)})$ over $M$. So $W$ is a subvariety of $V := V_0 \times V_k$. Let $\{T_0,\dots,T_s\} = \mathcal{T}(V_\alpha(x,z) \times V_k)$.

Note that $\cd(W) = \cd(V_0) = k_\alpha$. Indeed, since $M$ is algebraically closed and all the elements in the tuples $e_0,\dots,e_{\mu(\beta)}$ are from $M[a]$, we have:
\begin{align*}
\cd(W) 
&= \ldim(a,e_0,\dots,e_{\mu(\beta)}/M) - \trdeg(a,e_0,\dots,e_{\mu(\beta)}/M)\\
&= \ldim(a/M) - \trdeg(a/M)\\
&= \cd(V_0).
\end{align*}
 
Let $T$ be the minimal torus of $W$ and $m \in W(M)$ be such that $W \subset m T$.

Let $W'$ be a maximal irreducible subvariety of $V$ defined over $M$ with $W \subseteq W'$ and $\cd(W') = \cd(W)$. The variety $W'$ is hence $\cd$-maximal in $V$, its minimal torus is therefore some $T_{\tau} \in \mathcal{T}(V_\alpha(x,z) \times V_k)$, and $W' \subseteq m T_\tau$.

Let $(a^*,e_0^*,\dots,e_{\mu(\beta)}^*)$ be a generic of $W'$ over $M$. From $\cd(W') = \cd(W) = \cd(V_0)$,  we get
\begin{align*}
\ldim(a^*/M)-\trdeg(a^*/M)&= \ldim(a/M)-\trdeg(a/M)\\
&=\ldim(a^*,e_0^*,\dots,e_{\mu(\beta)}^*/M) -\trdeg(a^*,e_0^*,\dots,e_{\mu(\beta)}^*/M).
\end{align*}
Therefore, 
\[
\ldim(e_0^*,\dots,e_{\mu(\beta)}^*/Ma^*)=
\trdeg(e_0^*,\dots,e_{\mu(\beta)}^*/Ma^*)=:\ell.
\]

We now choose a linear basis $f_0,\dots,f_{\ell-1}$ over $Ma^*$ from the elements of the tuple $(e^*_0,\dots,e^*_{\mu(\beta)})$. The elements $f_0,\dots,f_{\ell-1}$ are hence algebraically independent over $Ma^*$. 

%%%%%%%%
That the coset $m T$ contains the green tuple $(a,e_0,\dots,e_{\mu(\beta)})$ implies that if 
we describe $m T$ by equations $\{\prod_{i=1}^n x_i^{n_{ij}}=c_j\}$, then the $c_j$ are all green. 
(Note that $c_j=\prod m_i^{n_{ij}}$.) As $mT_\tau\supseteq mT$, the same is true for a set of equations defining $mT_\tau$.
 It is then easy to see that there is a structure $N$ in $\Cl$ extending $M$ and with domain $N = \langle M a^*,f_0,\dots,f_{\ell-1}\rangle = \langle Ma^*,e^*_0,\dots,e^*_{\mu(\beta)}\rangle$ such that the tuple $(a^*,e^*_0,\dots,e^*_{\mu(\beta)})$ is green and $(a^*,f_0,\dots,f_{\ell-1})$ is a linear basis of $\U(N)$ over $\U(M)$.

For $0\leq j\leq \ell-1$, let $F_j:=\langle Ma^*, f_0,\ldots, f_{j-1}\rangle$ and observe that each extension $F_j\leq F_{j+1}$ is a green generic extension. 
Applying Lemma~\ref{L:8.3} repeatedly we get: $M[a^*]\in\Clm$ if and only if $N\in\Clm$.
Also, by Lemma~\ref{L:Kummer-strong2}, the map $a \mapsto a^*$ extends to an isomorphism 
over $M$ between $M[a]$ and $M[a^*]$. Thus, 
$$\text{$M[a]\in\Clm$ if and only if $N\in\Clm$}.$$
Now both $(e_0,\ldots,e_{\mu(\beta)})$ and $(e^*_0,\ldots,e^*_{\mu(\beta)})$ lie in $V_k$. And, since $(e_0,\ldots,e_{\mu(\beta)})$ is a specialisation of $(e^*_0,\ldots,e^*_{\mu(\beta)})$ and $(e_0,\ldots,e_{\mu(\beta)}) \not\in Z_k$, also $(e^*_0,\ldots,e^*_{\mu(\beta)}) \not\in Z_k$. Hence 
$(e^*_0,\ldots,e^*_{\mu(\beta)})$ realises $\psi_\beta$. 

Thus, in case (b), we have shown that the existence of a green difference sequence for $\beta$ of length $\mu(\beta)+1$ in $M[a]$ implies the existence of one such particular difference sequence in $N$, and, conversely, that the existence of such a difference sequence $(e^*_0,\ldots,e^*_{\mu(\beta)})$ in $N$ implies that $M[a]$ is not in $\Clm$.

Consider therefore the following condition on $b$: there is a tuple $m$ from $\U(M)$ and an irreducible component $W'$ of $V\cap m T_\tau$ (where $V$ is as above) such that:
\begin{enumerate}
\item If the coset $ m T_\tau$ is given by $\{\prod  x_i^{n_{ij}}=c_j\}_{1\leq j\leq t}$, then 
all $c_j$ are green;
\item $W'$ projects generically onto $V_0$;
\item $\cd(W')=\cd(V_0)$, and
\item for generic $(a^*,e^*_0,\ldots,e^*_{\mu(\beta)})$ in $W'$, 
$\psi_{\beta}(e^*_0,\ldots,e^*_{\mu(\beta)})$ holds.
\end{enumerate}
Note that the condition can be expressed by an existential sentence with parameters from $b$. Let $\theta^{(b)}(b)$ be the disjunction over all $\beta \in \Clc$ with $\mu^*(\beta) \leq n_\alpha$, all 
$k \in \{1,\dots,r\}$ and all $\tau \in \{1,\dots,s\}$ of these sentences.

Now $\chi_\alpha=\forall z\,[\neg\theta_\alpha(z)\vee\theta^{(a)}(z)\vee\theta^{(b)}(z)]$ is $\forall\exists$ and does 
the job.
\end{proof}

\bibliography{Bad-Torsion}

\end{document}